\documentclass[11pt,a4paper]{article}
\usepackage{a4wide,amsmath,amsthm,amssymb,amsfonts,color}
\usepackage{mathtools}\mathtoolsset{showonlyrefs}
\usepackage[all]{xy}
\usepackage{enumitem,authblk}

\usepackage{mathrsfs}
\newcommand{\cR}{{\mathcal R}}
\newcommand{\Q}{{\mathbb Q}}

\newcommand{\cO}{{\mathcal O}}

\newcommand{\gE}{{\mathfrak E}}
\newcommand{\gF}{{\mathfrak F}}

\newcommand{\gQ}{{\mathfrak Q}}
\newcommand{\gS}{{\mathfrak S}}

\newcommand{\rk}{\operatorname{rk}}

\renewcommand{\Im}{\operatorname{Im}}

 \newcommand{\gr}{\operatorname{Gr}}
\newcommand{\hgr}{\operatorname{\mathfrak{Gr}}}
\newtheorem{thm}{Theorem}[section]
\newtheorem{corol}[thm]{Corollary}
\newtheorem{prop}[thm]{Proposition}
\newtheorem{lem}[thm]{Lemma}
\newtheorem{defin}[thm]{Definition}
\newtheorem{conj}[thm]{Conjecture}
\theoremstyle{remark}
\newtheorem{rema}[thm]{Remark}
 \newenvironment{remark}{\begin{rema}}{\qee\end{rema}}
 \newcommand{\qee}{\ \hfill\hskip1pt $\triangle$}

\title{\bf\large FILTRATIONS OF NUMERICALLY FLAT HIGGS BUNDLES\\[5pt]
AND CURVE SEMISTABLE HIGGS BUNDLES\\[5pt] 
ON CALABI-YAU MANIFOLDS}

\author[1]{Ugo Bruzzo}
\author[2]{Armando Capasso}
\affil[1]{SISSA (Scuola Internazionale Superiore di Studi Avanzati),\par\vskip-3pt Via Bonomea 265, 34136 Trieste, Italy}
\affil[1]{INFN (Istituto Nazionale di Fisica Nucleare), Sezione di Trieste}
\affil[1]{IGAP (Institute for Geometry and Physics), Trieste}
\affil[1]{Arnold-Regge Center for Algebra, Geometry \par\vskip-3.5pt and Theoretical Physics, Torino}
\affil[2]{Via Tito Livio 17, 80040 Volla (NA), Italy}
\date{}
\begin{document}
\maketitle
%\begin{center}
%$^\S$ SISSA (Scuola Internazionale Superiore di Studi Avanzati), \\ Via Bonomea 265, 34136 Trieste, Italy  \\[3pt]
%$^\sharp$  Istituto Nazionale di Fisica Nucleare, Sezione di Trieste\\[3pt]
%$^\flat$ Arnold-Regge Center,  via P.~Giuria 1,  10125 Torino, Italy \\[3pt]
% $\ddag$   Institute for Geometry and Physics, Trieste
% \end{center}
 
% \bigskip
\begin{abstract} 
We consider Higgs bundles satisfying a notion of numerical flatness (H-nflatness) that was introduced in \cite{BHR,BG3}, and show that they have Jordan-H\"older filtrations whose quotients are stable, locally free and H-nflat. This is applied to show that curve semistable Higgs bundles on simply connected Calabi-Yau manifolds have vanishing discriminant. 
\end{abstract}

\let\svthefootnote\thefootnote
\let\thefootnote\relax\footnote{
\hskip-2.05\parindent {\em \hskip1mm Date: } \today
Email: {\tt bruzzo@sissa.it, armando1985capasso@gmail.com} \\
}
\addtocounter{footnote}{-1}\let\thefootnote\svthefootnote

\newpage

\section{Introduction}
Let $X$ be a smooth complex projective variety of dimension $n\ge 2$. We say that a vector bundle $E$ on $X$ is {\em curve semistable} if for every morphism $f\colon C \to X$, where $C$ is a smooth projective irreducible curve, the pullback $f^{*} E$ is semistable. After the results presented in \cite{Naka,BHR}, it turns out that curve semistable vector bundles can be characterized as follows.

\begin{thm} The following conditions are equivalent:
\begin{enumerate} \itemsep=-3pt
\item the vector bundle $E$ is curve semistable;
\item $E$ is semistable with respect to some polarization $H$, and
$\Delta(E) \cdot H^{n-2}= 0$, where $\Delta(E)$ is the characteristic class (discriminant)
$$ \Delta(E) = c_2(E) - \frac{r-1}{2r}c_1(E)^2 \in H^4(X,\Q)$$
and $r = \rk E$;
\item $E$ is semistable with respect to some polarization $H$, and $\Delta(E) = 0.$
\end{enumerate}
\end{thm}
 
\begin{remark} This Theorem implies that if a vector bundle has vanishing discriminant, then it is semistable with respect to a polarization if and only if it is semistable with respect to {\em all} polarizations.
\end{remark}
 
If $\gE=(E,\phi)$ is a Higgs bundle --- i.e., a vector bundle $E$ equipped with a morphism $\phi \colon E \to E\otimes\Omega^1_X$ such that the composition
\begin{equation}\label{phi2}
\phi\wedge \phi \colon E \xrightarrow{\phi} E\otimes\Omega^1_X
\xrightarrow{\phi\otimes\operatorname{id}} E\otimes\Omega^1_X \otimes\Omega^1_X
\to E \otimes\Omega^2_X
\end{equation}
vanishes --- we may say again that $\gE$ is curve semistable if all pullbacks $f^{*}\gE$ are semistable as Higgs bundles. While conditions (ii) and (iii) are equivalent also in this setting, and they imply condition (i), as it was proved in \cite{BHR,BG3}, it is not clear if condition (i) implies the others. Therefore we state this fact as a conjecture. 

\begin{conj} A curve semistable Higgs bundle $\gE$ has vanishing discriminant.
\label{theconj}
\end{conj}
(The fact that a curve semistable Higgs bundle $\gE$ is semistable with respect to any polarization is quickly proved using the\,\textquotedblleft easy\textquotedblright\,direction of the Mehta-Ramanathan restriction theorem, see e.g.~\cite{Si2}.)

After \cite{BLG}, we say that a smooth projective variety $X$ is a {\em Higgs variety} if this conjecture holds on it. What is known so far about the characterization of Higgs varieties is what follows:
\begin{itemize}\itemsep=-3pt
\item smooth projective varieties whose tangent bundle is numerically effective are Higgs \cite{BLG};
\item K3 surfaces are Higgs \cite{BLL};
\item varieties related to a Higgs variety by some standard geometric constructions, such as \'etale coverings, and others, are Higgs; see \cite{BLG} for details.
\end{itemize} 
The main result of this paper is to show that simply connected Calabi-Yau manifolds of all dimensions are Higgs varieties, generalizing \cite{BLL}, where this was proved for K3 surfaces. This will be accomplished by proving an equivalent version of the Conjecture \ref{theconj}, which is stated in terms of a class of Higgs bundles that have been called {\em Higgs numerically flat} (for short, H-nflat) \cite{BHR,BG3}; the main technical points are two. First we prove that H-nflat Higgs bundles have a particular kind of filtrations, and then we apply a result given in \cite{BBGL} which states that on a simply connected Calabi-Yau manifold a polystable Higgs bundle always has a vanishing Higgs field. 

We review the definition of H-nflat Higgs bundles. Let us remind that a line bundle $L$ on a smooth projective variety $X$ is said to be \emph{numerically effective} (\emph{nef}, for short) if for every irreducible curve $C$ in $X$ the inequality $c_1(L)\cdot C \geq0 $ holds. A vector bundle $E$ is said to be nef if the tautological line bundle $\mathcal{O}_{\mathbb P E}(1)$ on the projectivization $\mathbb{P}E$ is nef (we adopt the convention according to which the projectivization $\mathbb{P} E$ parameterizes rank one locally free quotients of $E$, in the sense described for instance in \cite{Fulton} or \cite{Har}). $E$ is said to be {\em numerically flat} if it is nef, and the dual bundle $E^\vee$ is nef as well. One can also introduce the {\em Grassmann bundles} associated with $E$: for every integer $s$ such that $ 0 < s < \rk E$, the variety $\gr_s(E)$ is a bundle over $X$, whose fibres are the Grassmann varieties of the fibres of $E$, which parameterizes the rank $s$ locally free quotients of $E$. Of course $\gr_1(E) = \mathbb PE$. Denoting by $\pi_s\colon \gr_s(E) \to X$ the projection, on each variety $\gr_s(E)$ there is a {\em universal quotient bundle} $Q_s$ of $\pi_s^{*} E$, which turns out to be numerically effective if $E$ is numerically effective.

In \cite{BHR} the {\em Higgs Grassmannians} $\hgr_s(\gE)$ of a Higgs bundle $\gE=(E,\phi)$ were introduced as suitably defined closed subschemes of the Grassmann bundles $\gr_s(E)$; again in the same sense, they parameterize rank $s$ locally free Higgs quotients of $\gE$. These Higgs Grassmannians are used to provide a notion of {\em Higgs numerically effective} Higgs bundle, i.e., a generalization of the notion of numerically effective vector bundles that is sensitive to the Higgs field, and again, a Higgs bundle is {\em Higgs numerically flat} (H-nflat) if both $\gE$ and its dual Higgs bundle $\gE^\vee$ are Higgs numerically effective.

Now Conjecture \ref{theconj} can be rephrased as follows:

\begin{conj} If $\gE=(E,\phi)$ is an H-nflat Higgs bundle, then $c_i(E)=0$ for all $i>0$.
\label{thesecondconj}
\end{conj}
(This fact is true for numerically flat vector bundles, see \cite{DPS}.) The  easy fact that    the two forms of the Conjecture are equivalent  is proved for instance in \cite{BLL}.

As we already anticipated, the main technical tool proved in this paper is the existence of a special kind of filtrations of H-nflat bundles: indeed, Theorem \ref{3.18ter} states that an H-nflat bundle on a smooth projective variety has a filtration (which is a Jordan-H\"older filtration) whose quotients are locally free, stable and H-nflat. Also results from \cite{BBG} will play an important role; there in particular it is proved that kernel and cokernel of a morphism of H-nflat Higgs bundles are themselves locally free and H-nflat.

In \cite{crelle} a notion of numerical effectiveness for Higgs bundles was given in terms of bundle metrics: a Higgs bundle $\gE$ which carries a metric which satisfies a suitable positiveness condition is said to be {\em $1$-H-nef}; and when that happens also for the dual bundle, we say that $\gE$ is {\em $1$-H-nflat}. While $1$-H-nefness can be shown to imply H-nefness, it is not clear whether the opposite implication is true as well (on the other hand, the two notions are equivalent for line bundles and for Higgs vector bundles on curves). In Section \ref{5} we shall show, as another application of the Theorem on the filtrations of H-nflat bundles, that on a Higgs variety H-nflat bundles are $1$-H-nflat, and actually that this property characterizes Higgs varieties (thus, the fact the notions of H-nflatness and 1-H-nflatness coincide is still another form of the Conjecture \ref{theconj}).

We conclude this introduction with the basic definitions about Higgs sheaves. Let $X$ be a smooth $n$-dimensional projective variety over the complex numbers, equipped with a polarization $H$. The degree of a coherent $\mathcal O_X$-module $F$ is the integer number 
$$\deg F=c_1(F) \cdot H^{n-1} $$
and if $F$ has positive rank, its slope is defined as 
$$\mu( F) = \frac{\deg F}{\rk F}. $$
 
\begin{defin} A Higgs sheaf on $X$ is a pair $\gE=(E,\phi)$, where $E$ is a coherent sheaf on $X$, and $\phi\colon E\to E\otimes \Omega_X^1$ is a morphism of $\mathcal O_X$-modules such that $\phi\wedge\phi=0$ (see eq.~\eqref{phi2}). A section $s$ of $E$ is $\phi$-invariant if there exists a section $\lambda$ of $\Omega_X^1$ such that $\phi(s)=s\otimes\lambda$. A Higgs subsheaf of a Higgs sheaf $(E,\phi)$ is a $\phi$-invariant subsheaf $G$ of $E$, i.e., $\phi(G) \subset G\otimes\Omega_X^1$. A Higgs quotient of $\gE$ is a quotient of $E$ such that the corresponding kernel is $\phi$-invariant. A Higgs bundle is a Higgs sheaf whose underlying coherent sheaf is locally free.

If $\gE=(E,\phi)$ and $\mathfrak G= (G,\psi)$ are Higgs sheaves, a morphism $f\colon \gE \to \mathfrak G $ is a homomorphism of $\mathcal O_X$-modules $f\colon E\to G$ such that the diagram
$$\xymatrix{
E\ar[r]^f \ar[d]_\phi & G \ar[d]^\psi \\
E\otimes\Omega^1_X \ar[r]^{f\otimes\operatorname{id}} & G\otimes\Omega^1_X }
$$
commutes.
\label{defHnef}
\end{defin}

\begin{defin} A torsion-free Higgs sheaf $ \gE=(E,\phi) $ is semistable (respectively, stable) if $\mu(G)\le \mu(E)$ (respectively, $\mu(G)< \mu(E)$) for every Higgs subsheaf $\mathfrak G=(G,\psi)$ of $ \gE$ with $0<\rk G < \rk E$. It is polystable if it is a direct sum of stable Higgs sheaves having the same slope.\end{defin}

For future use, we remind that semistable Higgs sheaves admit Jordan-H\"older filtrations; i.e., if $\gE$ is a semistable Higgs sheaf, there is a filtration in Higgs sheaves
$$ 0 = \gF_{m+1} \subset \gF_m \subset \dots \subset \gF_0 = \gE $$
whose quotients $\gF_i/\gF_{i+1}$ are stable and have all the same slope as $\gE$ \cite{Si2}. The associated graded module
$$\operatorname{Grad}(\gE) =\bigoplus_{i=0}^{m} \gF_i/\gF_{i+1}$$
is unique up to isomorphism.

\paragraph{Ackowledgements.} We thank Valeriano Lanza for useful discussions and suggestions.  Support for this research was provided by PRIN ``Geometria delle variet\`a algebriche'' and INdAM-GNSAGA. This paper was finalized while U.B.~was visiting the Department of Mathematics of Universidade Federal da Para\'iba, Jo\~ao Pessoa, Brazil.
 
\section{H-nflat Higgs bundles}
 
In this section we remind the main definitions concerning Higgs numerically flat Higgs bundles. Let $E$ be a rank $r$ vector bundle on a smooth projective variety $X$, and let $0<s<r$ an integer number. Let $\pi_s\colon\operatorname{Gr}_s(E) \to X $ be the Grassmann bundle parameterizing rank $s$ locally free quotients of $E$ \cite{Fulton}. In the exact sequence of vector bundles on $\operatorname{Gr}_s(E)$ 
\begin{equation}\label{univ}
\xymatrix{
0\ar[r] & S_{r-s,E}\ar[r]^(.55){\psi} & \pi_s^{*}E\ar[r]^(.45){\eta}\ar[r] & Q_{s,E}\ar[r] & 0
}
\end{equation}
$S_{r-s,E}$ is the universal rank $r-s$ subbundle of $\pi_s^{*} E$ and $Q_{s,E}$ is the universal rank $s$ quotient.
 
Let now $\gE\,=\,(E,\phi) $ be a rank $r$ Higgs bundle on $X$. One defines closed subschemes $\hgr_s(\gE)\subset\gr_s(E)$ as the zero loci of the composite morphisms
\begin{equation}\label{lambda}
(\eta\otimes\text{Id})\circ\pi_s^{*}(\phi)\circ\psi\,\colon\, S_{r-s,E}\to Q_{s,E}\otimes\pi_s^{*}\Omega_X^1.
\end{equation}
The restriction of \eqref{univ} to $\hgr_s(\gE)$ yields a universal exact sequence
\begin{equation}\label{g1}
\xymatrix{
0\ar[r] & \gS_{r-s,E}\ar[r]^(.55){\psi} & \rho_s^{*}\gE\ar[r]^(.45){\eta}\ar[r] & \gQ_{s,E}\ar[r] & 0,
}
\end{equation}
where $\mathfrak Q_{s,\gE}=Q_{s,E}\vert_{\hgr_s(\gE)}$ is equipped with the quotient Higgs field induced by the Higgs field $\rho_s^{*}\phi$ (here $\rho_s= \pi_s\vert_{\hgr_s(\gE)}\colon\hgr_s(\gE)\to X$).
The scheme $\hgr_s(\gE)$ enjoys the usual universal property: a morphism of varieties $f\colon T \rightarrow X$ factors through $\hgr_s(\gE)$ if and only if the pullback $f^{*}\gE$ admits a locally free rank $s$ Higgs quotient. In that case the pullback of the above universal sequence on $\hgr_s(E)$ gives the desired quotient of $f^{*}\gE$.

\begin{defin}\label{moddef}
A Higgs bundle $\gE=(E,\phi)$ of rank one is said to be Higgs numerically effective (H-nef for short) if $E$ is numerically effective in the usual sense. If $\rk\gE \geq 2$, we inductively define H-nefness by requiring that
\begin{enumerate}
\item all Higgs bundles ${\mathfrak Q}_{s,\gE}$ are H-nef for all $s$, and
\item the determinant line bundle $\det E$ is nef.
\end{enumerate}
If both $\gE$ and $\gE^{\vee}$ are Higgs numerically effective, $\gE$ is said to be Higgs numerically flat (H-nflat).
\end{defin}

Definition \ref{moddef} implies that the first Chern class of an H-nflat Higgs bundle is numerically equivalent to zero. Note that if $\gE\,=\,(E,\,\phi)$, with $E$ nef in the usual sense, then $\gE$ is H-nef. If $\phi\,=\,0$, the Higgs bundle $\gE\,=\,(E,\,0)$ is H-nef if and only if $E$ is nef in the usual sense.

The following statement generalizes Proposition 1.2 (12) in \cite{CP} to H-nef Higgs bundles; it will used below to prove the key Lemma of this paper. The proof is an easy adaptation of that in \cite{CP}.
\begin{prop}\label{keyprop}
Let $\gE=(E,\phi)$ be an H-nef Higgs bundle on $X$ and let  $\gE^{\vee}=(E^{\vee},\phi^{\vee})$ be the dual Higgs bundle. If $s$ is a $\phi^{\vee}$-invariant section of $E^{\vee}$, then $s$ has no zeroes.
\end{prop}
\begin{proof} 
Note that $s$ defines a monomorphism of Higgs sheaves $f\colon (\cO_X,\lambda)\to(E^{\vee},\phi^{\vee})$, where $\phi^{\vee}(s)=s\otimes\lambda$ and $\lambda\in H^0(X,\Omega^1_X)$.  Dualizing this monomorphism, one has a morphism of Higgs sheaves $f^{\vee}:(E,\phi)\to(\cO_X,\lambda)$; if $s$ has zeroes,     then $f^{\vee}$ has zeroes as well, and $\Im f^{\vee}$ is a proper Higgs subsheaf of $(\cO_X,\lambda)$, hence it has negative degree on some curve in $X$. This contradicts the fact that $\gE$  is H-nef.
\end{proof}

\section{Filtering H-nflat Higgs bundles}

In this section, $X$ will be a smooth projective variety over the complex numbers,
equipped with a fixed polarization, which in particular will be used to compute the degrees of the various coherent sheaves considered.

The next Lemma is the key technical result of this paper.

\begin{lem}\label{3.18bis}
Let $\gE=(E,\phi)$ be an H-nflat Higgs bundle on $X$ of rank $r \ge 2$. If $\gE$ is not stable, it can be written as an extension
\begin{equation}\label{ses}
\xymatrix{
0\ar[r] & \gF\ar[r] & \gE\ar[r] & \gQ\ar[r] & 0
}
\end{equation}
where $\gF$ and $\gQ$ are locally free H-nflat Higgs bundles, and $\gF$ is stable.%\footnote{In general $Q$ is just semistable, as it is H-nflat.}
\end{lem}

\begin{proof}
Note that $\gE$ is semistable by Proposition~A.8 of \cite{BG3} and has degree zero. Let $\gF=(F,\psi) $ be a Higgs subsheaf of $E$ of rank $p$, with $0 < p < r$. As $\gE$ is semistable of zero degree, $\bigwedge^p\gE$ is semistable of zero degree as well \cite[Corollary ~3.8]{S92}. Let $\det F=\left(\bigwedge^pF\right)^{\vee\vee}$ be the determinant of $F$, and let $\det \gF$ be the sheaf $\det F$ equipped with the naturally induced Higgs field. As $\det \gF$ injects into $\bigwedge^p\gE$ (as Higgs sheaf), we have $\deg F \le 0$. 

We can assume that $\gF$ is a reflexive Higgs subsheaf of $\gE$ of minimal rank $p>0$ with $\deg F= 0$. Then $\gF$ is stable. We have an exact sequence
\begin{equation} \label{mu}
\xymatrix{
0\ar[r] & \det\gF\ar[r] & \bigwedge^p\gE\ar[r] & \mathfrak R\ar[r] & 0
}
\end{equation}
where $\mathfrak R= (R,\chi)$ is the quotient Higgs sheaf. We use this to show that $(\det F)^\vee$ is nef. Let $f\colon C \to X$ be a morphism, where $C$ is a smooth projective irreducible curve. Then $f^{*} R$ splits as $\tilde R \oplus T$, where $\tilde R$ is locally free and $T$ is torsion. It is easy to check that $T$ with the restriction of the pullback Higgs field is a Higgs sheaf. Then $\tilde R$, again with the restriction of the pullback Higgs field, is a Higgs bundle, and is a quotient of $f^{*}\left(\bigwedge^p\gE\right)$; therefore it is H-nef, and then $\deg f^{*} R\ge 0$. Then $\deg (f^{*} \det F )\le 0$, and since the choice of $C$ is arbitrary, $(\det F)^\vee$ is nef. 

Now by Lemma 3.13 in \cite{crelle}\footnote{Actually that result was already contained in the proof 
Corollary 1.19 in \cite{DPS}.} one has $c_1(F)=0$, and by Proposition 1.2.(9) in \cite{CP} $\det(F)$ is numerically flat. 
Tensoring the exact sequence \eqref{mu} by $\det^{-1}\gF$ one obtains a
$\det\psi^{\vee}\otimes\phi^p$-invariant section $\sigma\colon (\cO_X,\lambda)\to\left(\det\gF\right)^{\vee}\otimes\bigwedge^p\gE$, where $\phi^p$ is the Higgs field of $\displaystyle\bigwedge^p\gE$ and $\left(\det\psi^{\vee}\otimes\phi^p\right)(\sigma)=\sigma\otimes\lambda$. By Proposition \ref{keyprop}, $\sigma$ has no zeroes, that is, $\det\gF$ is a Higgs subbundle of $\displaystyle\wedge^p\gE$; by Lemma 1.20 in \cite{DPS} $\gF$ is a Higgs subbundle of $\gE$. From all this, $\gF$ is an H-nflat Higgs bundle, and then by Proposition 3.7 in \cite{BBG}   the quotient Higgs sheaf $\gQ$ is locally free and H-nflat as well.
\end{proof}

\begin{thm}\label{3.18ter}
An H-nflat Higgs bundle $\gE=(E,\phi)$ on a smooth projective variety $X$ is pseudostable (i.e., it has a filtration whose quotients are locally free and stable), and moreover the quotients of the filtration are H-nflat.
\end{thm}
 
\begin{proof} We use Lemma \ref{3.18bis} as the basis for an iterative proof. Note that in eq.~\eqref{ses} the Higgs bundle $\gQ$, if it is not stable, satisfies the same hypotheses as $\gE$, so that it sits in an exact sequence
\begin{equation*}
\xymatrix{
0\ar[r] & \gQ_1\ar[r] & \gQ\ar[r] & \gQ_2\ar[r] & 0 
}
\end{equation*}
where $\gQ_1$ and $\gQ_2$ are locally free and H-nflat and $\gQ_1$ is stable. By the snake Lemma we have a diagram
\begin{equation} \xymatrix{
 & & & 0 \ar[d]\\
 & 0\ar[d] & 0\ar[d] & \gQ_1\ar[d]\\
0\ar[r] & \gF \ar[r]\ar[d] & \gE \ar[r]\ar[d] & \gQ \ar[r]\ar[d] & 0 \\
0\ar[r] & \gF_1 \ar[r]\ar[d] & \gE \ar[r]\ar[d] & \gQ_2 \ar[r]\ar[d] & 0\\
 & \gQ _1 \ar[d] & 0 & 0 \\
 & 0}
\end{equation} 
Note that again $\gF_1$ is locally free and H-nflat by Proposition~3.7 in \cite{BBG}. Now
$$ 0 \subset \gF \subset \gF_1 \subset \gE $$
is a filtration whose quotients $\gQ_1$ and $\gQ_2$ are locally free and H-nflat; moreover, $\gF$ and $\gQ_1$ are stable. If $\gQ_2$ is stable as well, the claim is proved. If it is not, we iterate the procedure, until we get a quotient which is stable (possibly a line bundle). At step $k$ we shall have the diagram
\begin{equation} \xymatrix{
&& & 0 \ar[d]\\
 & 0\ar[d] & 0\ar[d] & \gQ_{k}\ar[d]\\
0\ar[r] & \gF_{k-1}\ar[r]\ar[d] & \gE \ar[r]\ar[d] & \gQ_{k-1} \ar[r]\ar[d] & 0 \\
0\ar[r] & \gF_k \ar[r]\ar[d] & \gE \ar[r]\ar[d] & \gQ_{k+1} \ar[r]\ar[d] & 0\\
 & \gQ _{k} \ar[d] & 0 & 0 \\
& 0}
\end{equation} 
and if $n$ is the last step we get a filtration
\begin{equation}\label{JH} 0 \subset \gF \subset \gF_1 \subset \dots \subset \gF_n \subset \gE\end{equation}
whose quotients $\gQ_1,\dots,\gQ_{n+1}$ are locally free, stable and H-nflat.
\end{proof}
 
\begin{corol} If $\gE=(E,\phi)$ is an H-nflat Higgs bundle on a smooth projective variety $X$ such that all the quotients  of the filtration \eqref{JH} have rank 1, then $c_i(E)=0$.\footnote{Heuristically, these are the H-nflat Higgs bundles that are the farthest from being stable.}
\end{corol}
\begin{proof}
Indeed $c_1(Q_k)=0$ for all $k\ge 1$ as each $\gQ_k$ is an H-nflat line bundle, so that $c_i(E)=0$ for all $i$.
\end{proof}

\begin{remark}
By the uniqueness up to isomorphism of the Jordan-H\"older filtration, \eqref{JH} is a Jordan-H\"older filtration of $\gE$. 
\end{remark}

\section{The conjecture for Calabi-Yau manifolds}

We prove that simply connected Calabi-Yau manifolds are Higgs varieties by proving the Conjecture in the form of the Conjecture \ref{thesecondconj}.
\begin{thm} If $X$ is a simply connected Calabi-Yau variety, and $\gE=(E,\phi)$ is an H-nflat Higgs bundle
on it, then $c_i(E)=0$ for all $i\ge 1$.
\end{thm}

\begin{proof}
To compute the Chern classes of $\gE$, we can replace it with the graded module associated  to the filtration \eqref{JH}, i.e., we can assume that $\gE$ is polystable and H-nflat. Then by the main result in \cite{BBGL} (Corollary 2.6), $\phi=0$. So $E$ is actually numerically flat as a vector bundle, and then $c_i(E)=0$ \cite{DPS}.
\end{proof}
This generalizes Theorem 6.4 of \cite{BLL}, where this result was proved for K3 surfaces.

\begin{remark}
Corollary 2.6 in  \cite{BBGL} states that if $\gE=(E,\phi)$ is a polystable Higgs bundle on a simply connected Calabi-Yau manifold, then $\phi=0$. If $\gE$ is semistable, then the Higgs field on the graded module of the Jordan-H\"older filtration of $\gE$ vanishes, but $\phi$ itself may be nonzero (examples are given in \cite{BLL}).
\end{remark}

\section{Higgs varieties and special metrics}\label{5}

In \cite{crelle}, following the work of De Cataldo \cite{DeC} for ordinary vector bundles, a notion of numerical effectiveness for Higgs bundles was given in terms of bundle metrics. If $\gE =(E,\phi)$ is a Higgs bundle, and $h$ is an Hermitian metric on $E$, one defines the {\em Hitchin-Simpson connection} of the pair $(\gE,h)$ as
$$ \mathcal D_{(h,\phi)} = D_ h + \phi + \bar \phi $$
where $D_h$ is the Chern connection of the Hermitian bundle $(E,h)$, and $\bar\phi$ is the metric adjoint of $\phi$ defined as 
\begin{displaymath}
h(s,\phi(t))=h\left(\overline{\phi}(s),t\right)
\end{displaymath}
for all sections $s,t$ of $E$. The curvature $\cR_{(h,\phi)}$ of the Hitchin-Simpson connection defines a bilinear form on $T_X\otimes E$, where $T_X$ is the tangent bundle to $X$, by letting
\begin{equation}\label{curvatilde} 
\widetilde{\cR}_{(h,\phi)}(u\otimes s,v\otimes t)=\frac{i}{2\pi}\left\langle h\left({\cR}_{(h,\phi)}^{(1,1)}(s),t\right),u\otimes v\right\rangle\,.
\end{equation}
where ${\cR}^{(1,1)}_{(h,\phi)}$ is the $(1,1)$-part of $\cR_{(h,\phi)}$, and $\langle\,,\rangle$ is the scalar product given by the K\"ahler form associated with the given polarization of $X$.
 
\begin{defin} A Higgs bundle $\gE=(E, \phi)$ on $X$ is said to be
\begin{enumerate}
\item $1$-H-nef if for every $\xi > 0$ there exists an Hermitian
metric $h_{\xi}$ on $E$ such that the bilinear form $$\widetilde{\cR}_ {\left(\gE,h_{\xi}\right)}  + \xi \omega \otimes h_\xi$$
is semipositive definite on all
sections of $T_X\otimes E$ that, at every point $x$ in their domain, define a rank one tensor in  $(T_X)_x\otimes E_x$; 
\item $1$-H-nflat if both $\gE$ and $\gE^{\vee}$ are $1$-H-nef.
\end{enumerate}
\end{defin}
 
It was shown in \cite{crelle} that   $1$-H-nef Higgs bundles are H-nef. The opposite implication is known to hold for Higgs line bundles, and for Higgs bundles on curves; it is unknown whether it holds in general. This fact is related to the Conjectures \ref{theconj} and \ref{thesecondconj}; indeed, it was shown in \cite{crelle} that Conjecture \ref{thesecondconj} holds if\,\textquotedblleft H-nflat\textquotedblright\,is replaced by\,\textquotedblleft $1$-H-nflat\textquotedblright.
 
Actually Theorem \ref{3.18ter} implies the following result.
 
\begin{thm} \label{anotherconj} The following conditions are equivalent. 
\begin{enumerate}\itemsep=-3pt
\item Every H-nflat Higgs bundle on the projective variety $X$ is $1$-H-nflat.
\item $X$ is a Higgs variety. 
\end{enumerate}
\end{thm}
\begin{proof}
The fact that condition (i) implies (ii) follows from the previous discussion. To prove the opposite implication, let $\gE=(E,\phi)$ be an H-nflat Higgs bundle. The quotients $\gQ_k=(Q_k,\phi_k)$ of the filtration in Theorem \ref{3.18ter} are H-nflat, and since $X$ is a Higgs variety, they have vanishing Chern classes; moreover, as they are stable, they carry {\em Hermitian-Yang-Mills} metrics, i.e., on each bundle $Q_k$ there is an Hermitian metric $h_k$ such that the mean curvature $\mathcal K_{\left(h_k,\phi_k\right)}$ of the Hitchin-Simpson connection vanishes \cite{Si-var}. Then we have
\begin{multline}
0= -4\pi^2 \operatorname{ch}_2(Q_k) = 
\int_X\operatorname{tr}\left(\cR_{\left(h_k,\phi_k\right)}\wedge\cR_{\left(h_k,\phi_k\right)}\right)\wedge\omega^{n-2}= \\ = 
\gamma_{1,k}\left\|\cR_{\left(h_k,\phi_k\right)}\right\|^2
-\gamma_{2,k}\left\|\mathcal K_{\left(h_k,\phi_k\right)}\right\|^2
=\gamma_{1,k}\left\|\cR_{\left(h_k,\phi_k\right)}\right\|^2
\end{multline}
for some positive constants $\gamma_{1,k}$ and $\gamma_{2,k}$, where the norms are $L^2$-norms, and $\omega$ is the K\"ahler form. So all the Hitchin-Simpson curvatures of the Higgs bundles $\gQ_k$ vanish. Finally, 
Theorem 3.16 of \cite{crelle} implies that $\gE$ is $1$-H-nflat.
\end{proof}
 
 \begin{remark} If we set the Higgs field to zero in Theorem \ref{anotherconj}, i.e., if we apply the Theorem
 to ordinary bundles, we obtain that the notions of 1-numerical flatness and numerical flatness coincide.
 \end{remark}


\bigskip
\frenchspacing

\begin{thebibliography}{10}

%\bibitem{Barth-stable}
%{\sc W. Barth}, {\em Some properties of stable rank-{$2$} vector bundles on
%  {${\bf P}_{n}$}}, Math. Ann., 226 (1977), pp. 125--150.

\bibitem{BBGL} {\sc I. Biswas, U. Bruzzo, B. Gra\~na Otero, and A. Lo~Giudice}, {\em Yang-{M}ills-{H}iggs connections on {C}alabi-{Y}au manifolds}, Asian J. Math. 20 (2016) 989--1000.

\bibitem{BBG}
{\sc I. Biswas, U. Bruzzo, and S. Gurjar}, {\em Higgs bundles and fundamental
  groups schemes}.
\newblock {\tt arXiv:1607.07207 [math.AG]}. To appear in Adv. Geom.

\bibitem{crelle} {\sc U. Bruzzo and B. Gra\~na~Otero,} \emph{Metrics on semistable and numerically effective Higgs bundles}, J. reine ang. Math. 612 (2007) 59--79.

\bibitem{BG3}
---, {\em Semistable and numerically
  effective principal ({H}iggs) bundles}, Adv. Math. 226 (2011) 3655--3676.

\bibitem{BHR}
{\sc U. Bruzzo and D. Hern{\'a}ndez~Ruip{\'e}rez}, {\em Semistability vs.
  nefness for ({H}iggs) vector bundles}, Diff. Geom. Appl., 24 (2006),
  pp. 403--416.

\bibitem{BLL} 
{\sc U. Bruzzo, A. Lo Giudice and V. Lanza,} {\em Semistable Higgs bundles on Calabi-Yau manifolds,}
{\tt arXiv:1710.03671 [math.AG]}. To appear in Asian J. Math.

%\bibitem{BSchYM}
%{\sc I. Biswas and G. Schumacher}, {\em Yang-{M}ills equation for stable
%  {H}iggs sheaves}, Internat. J. Math., 20 (2009), pp. 541--556.
%

%\bibitem{BHR}
%{\sc U. Bruzzo and D. Hern{\'a}ndez~Ruip{\'e}rez}, {\em Semistability vs.
%  nefness for ({H}iggs) vector bundles}, Differential Geom. Appl., 24 (2006),
%  pp. 403--416.
%
\bibitem{BLG}
{\sc U. Bruzzo and A. Lo Giudice}, {\em Restricting {H}iggs bundles to curves},
  Asian J. Math., 20 (2016), pp. 399--408.

\bibitem{CP}
{\sc F. Campana and T. Peternell}, {\em Projective manifolds whose tangent bundles are numerically effective}, Math. Ann. 289 (1991) 169--187

\bibitem{DeC} {\sc M. A. A. De Cataldo}, {\em Singular Hermitian metrics on vector bundles}, J. reine ang. Math. 502 (1998) 93--122.  

\bibitem{DPS}
{\sc J.-P. Demailly, T. Peternell, and M. Schneider}, {\em Compact complex
  manifolds with numerically effective tangent bundles}, J. Algebraic Geom., 3
  (1994), pp. 295--345.
  
  \bibitem{Fulton} {\sc W. Fulton,} {\em Intersection theory}, Ergebnisse der Mathematik und ihrer Grenzgebiete (3), vol. 2, 
Springer-Verlag, Berlin 1998. 

\bibitem{Har}
{\sc R. Hartshorne,} {\em Algebraic geometry,}
Graduate Texts in Mathematics, vol. 52, Springer-Verlag, New York 1977.
 

%\bibitem{Ko}
%{\sc S. Kobayashi}, {\em Differential geometry of complex vector bundles},
%  vol. 15 of Publications of the Mathematical Society of Japan, Princeton
%  University Press, Princeton, NJ, 1987.
%\newblock Kan{\^o} Memorial Lectures, 5.
%
%\bibitem{BRcong}
%{\sc V. Lanza and A. Lo~Giudice}, {\em Bruzzo's conjecture}, J. Geom. Physics,
%  118 (2017), pp. 181--191.
%
%\bibitem{Mi}
%{\sc Y. Miyaoka}, {\em The {C}hern classes and {K}odaira dimension of a minimal
%  variety}, in Algebraic geometry, Sendai, 1985, vol. 10 of Adv. Stud. Pure
%  Math., North-Holland, Amsterdam, 1987, pp. 449--476.
%
\bibitem{Naka}
{\sc N. Nakayama}, {\em Normalized tautological divisors of semi-stable vector
  bundles}, S\=urikaisekikenky\=usho K\=oky\=uroku,  (1999), pp. 167--173.
\newblock Kyoto University, Research Institute for Mathematical Sciences.

%\bibitem{OSS}
%{\sc C. Okonek, M. Schneider, and H. Spindler}, {\em Vector bundles on complex
%  projective spaces}, vol. 3 of Progress in Mathematics, Birkh\"auser, Boston,
%  Mass., 1980.
%
%\bibitem{S88}
%{\sc C. T. Simpson}, {\em Constructing variations of {H}odge structure using
%  {Y}ang-{M}ills theory and applications to uniformization}, J. Am. Math. Soc.,
%  1 (1988), pp. 867--918.
%

\bibitem{Si-var}
{\sc C. T. Simpson}, {\em Constructing variations of {H}odge structure using
  {Y}ang-{M}ills theory and applications to uniformization}, J. Am. Math. Soc.,
  {\bf 1} (1988), pp. 867--918.
  
\bibitem{S92}
---,  {\em {H}iggs bundles and
  local systems}, Inst. Hautes \'Etudes Sci. Publ. Math.,  75 (1992) 5--95.
  
\bibitem{Si2}  ---,
{\em Moduli of representations of the fundamental group of a smooth projective
	variety. {I}},
Inst. Hautes \'Etudes Sci. Publ. Math.,  79 (1994) 47--129.


  
%  \bibitem{Yau77}
%  {\sc S.T. Yau,} {\em Calabi's conjecture and some new results in algebraic
%              geometry,}  Proc. Nat. Acad. Sci. US  74 (1977) 1798--1799.

\end{thebibliography}
\end{document}